\documentclass{amsart}
\usepackage[utf8]{inputenc}          
\DeclareUnicodeCharacter{200E}{}     
\newtheorem{theorem}{Theorem}[section]
\newtheorem{lemma}[theorem]{Lemma}

\theoremstyle{definition}
\newtheorem{definition}[theorem]{Definition}

\theoremstyle{remark}
\newtheorem{remark}[theorem]{Remark}

\numberwithin{equation}{section}

\newcommand{\Ir}{{\bf{Ir}}}
\newcommand{\V}{{\bf{V}}}

\newcommand{\I}{{\bf{I}}}
\newcommand{\U}{{\bf{U}}}

\begin{document}

\title[]{An Algebraic Proof of the Polynomial Version of van der Waerden's Theorem}
\author{ J. Jafari}
\curraddr{Department of Pure Mathematics,
Faculty of Sciences, University of Guilan}
\email{javadjafri66@gmail.com}

\author{M. A. Tootkaboni }
\curraddr{Department of Pure Mathematics,
	Faculty of Sciences, University of Guilan}
\email{tootkaboni.akbari@gmail.com}






\begin{abstract}
The polynomial version of van der Waerden's theorem, proved using dynamical systems by V. Bergelson and A. Leibman in 1996, \cite{Bergelson1996}, significantly highlighted the role of dynamical systems in addressing problems related to monochromatic configurations within algebraic structures. In this paper, by introducing symbolic polynomials, we aim to provide an alternative proof of the polynomial version of van der Waerden's theorem relying solely on Stone-\v{C}ech compactification of an infinite discrete semigroup.
\end{abstract}

\maketitle



	\section{Introduction}
	
	Van der Waerden's Theorem is a foundational result in Ramsey theory, first introduced in 1927 by van der Waerden in the context of proper colorings. The classical form of the theorem states that in any finite coloring of the natural numbers, there exists a monochromatic arithmetic progression of a given length. While the statement may not seem obvious at first glance, it served as a pivotal point for the growth and development of modern combinatorics, particularly in the area of coloring and random structures.
	
	\begin{theorem}[Van der Waerden,\cite{vandw1927}]
		For every pair of natural numbers $r$ and $k$, if the natural numbers is partitioned into $r$ colors, then at least one of the color classes contains an arithmetic progression of length $k$.
	\end{theorem}
	
	In the 1990s, a polynomial version of Van der Waerden's Theorem was introduced and proved by Bergelson and Leibman. This generalization replaces linear progressions with sequences generated by integer-valued polynomials without constant terms.
	
	\begin{theorem}[Polynomial van der Waerden, Bergelson–Leibman, \cite{Bergelson1996}]		
	Let $P_1(n)$, $P_2(n)$, $\dots,$ $P_k(n)$ be integer-valued polynomials such that $P_i(0) = 0$ for all $1 \leq i \leq k$. Then, for any finite coloring of the natural numbers, there exist $x$ and $n$ such that the set $\{x + P_1(n), x + P_2(n), \dots, x + P_k(n)\}$ is monochromatic.
	\end{theorem}
	
	 The original proof relied heavily on tools from ergodic theory and dynamical systems, highlighting the theoretical depth and the necessity of advanced methods in understanding rich coloring behaviors.
	
	Since then, various alternative proofs, generalizations, and versions have emerged, including more combinatorial proofs, multidimensional and topological variants.
	
	 The primary goal of this work is to provide a method based solely on pure combinatorial tools, while still preserving the original theorem's proving strength.

In the second part of this paper, we briefly review Stone-\v Cech compactification, partial semigroups, and some of their properties. Subsequently, in Chapter 3, we introduce the space of symbolic polynomials, which plays a central role in this work. This chapter consists of two subsections devoted to the introduction of one variable and multi variable symbolic polynomials, respectively. In this section, we show that every symbolic polynomial is associated with a polynomial in $\mathbb{Z}[x]$. This correspondence facilitates establishing a connection between the structure of color patterns in the integers and the space of symbolic polynomials. Finally, we present the proof of van der Waerden's theorem for symbolic polynomials, and we prove the main theorems as a consequence.

\section{Preliminary}
In this section, we state some essential concepts that are necessary for our work. First, we concentrate on the concepts of the Stone-\v Cech compactification and the adequate partial semigroup. For more details, see \cite{Hindman},\cite{McCutcheon}, \cite{McLeod} and \cite{McLeod1}.

\subsection{The Stone-\v Cech Algebra} 
For \(A \subseteq S\), we define $\overline{A} = \{p \in \beta S \mid A \in p\}$. The collection \(\{\overline{A} \mid A \subseteq S\}\) forms a basis for a topology on \(\beta S\), and $\beta S$ is a compact Hausdorff space respect the topology and $\beta S$ is called the Stone-$\check{C}$ech compactification of \(S\).

The operation \("\cdot"\) can be uniquely extended to \(\beta S\) and \((\beta S, \cdot)\) becomes a compact right topological semigroup. This means that for any \(p \in \beta S\), the function \(r_p \colon \beta S \to \beta S\) defined by \(r_p(q) = q \cdot p\) is continuous, with \(S\) contained in its topological center, i.e., for any \(x \in S\), the function \(\lambda_x \colon \beta S \to \beta S\) defined by \(\lambda_x(q) = x \cdot q\) is continuous. For \(p, q \in \beta S\) and \(A \subseteq S\), \(A \in p \cdot q\) if and only if $\{x \in S \mid x^{-1} \cdot A \in q\} \in p$,  where \(x^{-1}  A = \{y \in S \mid x \cdot y \in A\}\).

A nonempty subset \(I\) of a semigroup \((S, \cdot)\) is called a {left ideal} of \(S\) if \(\{sa:s\in S,a\in I\}=S \cdot I \subseteq I\), a {right ideal} if \(I \cdot S \subseteq I\), and a {two-sided ideal} (or simply an {ideal}) if it is both a left and a right ideal. A {minimal left ideal} is a left ideal that does not contain any proper left ideal. Similarly, we can define a {minimal right ideal}. An element $x\in S$ is an idempotent if and only if $ xx = x$ and, is called a minimal idempotent if $x$ belongs to a minimal left ideal.

\subsection{Adequate Partial Semigroups} 

We begin by introducing several fundamental concepts essential to our work. First, we focus on the notion of partial semigroup. Further details can be found in \cite{Hindman,McCutcheon}.  Let $S$ be a non-empty set, and let $*$ be a binary operation defined binary operation with domain $ D \subseteq S \times S $. The pair \((S, *)\) is called a \emph{partial semigroup} if, for all $ x, y, z \in S $, the associativity condition $(x * y)* z = x * (y * z)$ holds in the sense that if either side is defined, then the other is also defined and they are equal.

We say that $ x* y$ is defined if \( (x, y) \in D \). For every \( x \in S \), we define $R_S(x) = \{s \in S : x * s \hbox{ is defined }\}$  
and $L_S(x) = \{s \in S : s * x \hbox{ is defined}\}$. 
For \( s \in S \), we write \( S * s = \{t * s : t \in S\} \). Note that $S * s = L_S(s) * s $, and the notation should not cause confusion. In fact, we have $S* s=L_S(s)* s$. A nonempty subset \( I \subseteq S \) is called a \emph{left ideal} of $ S $ if $ y * x \in I $ for all $ x \in I $ and $ y \in L_S(x) $. Similarly, $ I \subseteq S $ is called a \emph{right ideal} if $ x * y \in I $ for all $ y \in R_S(x) $ and $ x \in I $. We say that $ I $ is an \emph{ideal} if it is both a left ideal and a right ideal. A subset $ L \subseteq S $ is called a \emph{minimal left ideal} if $ L $ is a left ideal of $ S $, and whenever $ J\subseteq L $ is a left ideal of $ S $,  we have $ J = L $. The notion of a minimal right ideal is defined analogously. An element $ p \in S $ is called an \emph{idempotent} if $ p * p = p $, and the collection of all idempotents is denoted by $ E(S) $.

We now recall some basic properties of partial semigroups; see, for instance, \cite{Hindman}.
\begin{definition}
	Let \((S, *)\) be a partial semigroup.
	\begin{itemize}
		\item[(a)] For \( H \in P_f(S) \), we define \( R_S(H) = \bigcap_{s \in H} R_S(s) \).	
		\item[(b)] We say that \((S, *)\) is \emph{adequate} if \( R_S(H) \neq \emptyset \) for all \( H \in P_f(S) \).
		\item[(c)] Define $\delta S = \bigcap_{H \in P_f(S)} \overline{R_S(H)}$.
	\end{itemize}
\end{definition}

The set $\delta S \subseteq \beta S$ is a compact right topological semigroup, as shown in Theorem 2.10 of \cite{McCutcheon}. If \((S, *)\) is a partial semigroup, for each \( s \in S \) and \( A \subseteq S \), we define 
\[
s^{-1}A = \{t \in R_S(s) : s * t \in A\}.
\]

\begin{definition}
	Let \( (S, *) \) be a right partial semigroup.
	\begin{itemize}
		\item[(a)] For \( a \in S \) and \( q \in \overline{R_S(a)} \), define
		\[
		a * q = \{A \subseteq S \colon a^{-1}A \in q\}.
		\]
		\item[(b)] For \( p, q \in \beta S \), define
		\[
		p * q = \{A \subseteq S \colon \{a \in S \colon a^{-1}A \in q\} \in p\}.
		\]
	\end{itemize}
\end{definition}
By Lemma 2.7 in \cite{McCutcheon}, if \( (S, *) \) is an adequate partial semigroup, then for \( a \in S \) and \( q \in \overline{R_S(a)} \), we have \( a * q \in \beta S \). Moreover, for \( p \in \beta S \), \( q \in \delta S \), and \( a \in S \), it holds that \( R_S(a) \in p * q \) whenever \( R_S(a) \in p \). Additionally, for every \( p, q \in \delta S \), we have \( p * q \in \delta S \).
\begin{lemma} \label{1.10} 
	Let \( T \) be an adequate partial semigroup, and let \( S \subseteq T \) be an adequate partial semigroup under the operation inherited from \( T \). Then the following statements are equivalent:
	\begin{itemize}
		\item[(a)] \( \delta S \subseteq \delta T \).
		\item[(b)] For all \( y \in T \), there exists \( H \in P_f(S) \) such that  
		\( \bigcap_{x \in H} R_S(x) \subseteq R_T(y) \).
		\item[(c)] For all \( F \in P_f(T) \), there exists \( H \in P_f(S) \) such that  
		\( \bigcap_{x \in H} R_S(x) \subseteq \bigcap_{x \in F} R_T(x) \).
	\end{itemize}
\end{lemma}

\begin{proof}
	See \cite{McLeod}.
\end{proof}

\begin{definition} \label{1.11}
	Let \( T \) be a partial semigroup. A subset \( S \subseteq T \) is called an \textbf{adequate partial subsemigroup} of \( T \) if \( S \) is an adequate partial semigroup under the operation inherited from \( T \), and for every \( y \in T \), there exists \( H \in P_f(S) \) such that
	\[
	\bigcap_{x \in H} R_S(x) \subseteq R_T(y).
	\]
\end{definition}

\begin{theorem} \label{1.23}
	Let \( T \) be an adequate partial semigroup, and let \( S \) be an adequate partial subsemigroup of \( T \). Assume further that \( S \) is an ideal of \( T \). Then \( \delta S \) is an ideal of \( \delta T \). In particular, \( K(\delta S) = K(\delta T) \).
\end{theorem}
\begin{proof}
	See \cite{McLeod1}.
\end{proof}

\section{\bf Symbolic Polynomials Space}
In this section, we introduce the space of symbolic polynomials. We will show that every polynomial with integer coefficients is the image of a symbolic polynomial. We will also see that every finite partition of the integers induces a finite partition of the space of symbolic polynomials.
\subsection{\bf{Symbolic one variable Polynomial Space}}

For \( k \in \mathbb{N} \), we fix the symbols \( 1_0, 1_1, \ldots, 1_k \), and construct strings of the form \( (1_0)( 1_{1})( 1_{2}) \ldots (1_{i}) \), where $0<i\leq k$. 
\begin{definition}
	
	(a) \( \Gamma_k = \{ (a_0 1_0)(a_1 1_1) \cdots (a_i 1_i) : i \in [1, k] \text{ and for } t \in [0,i],\ a_t \in \mathbb{Z} \} \), where $[j,i]=\{j,j+1,\ldots,i\}$ for every $j<i$ and $i,j\in\mathbb{Z}$.
	
	(b) If \( x = (a_0 1_0)(a_1 1_1) \cdots (a_i 1_i) \) and \( y = (b_0 1_0)(b_1 1_1) \cdots (b_j 1_j) \), are members of \( \Gamma_k \), then 
	$
	x = y $ if and only if $i = j$ and for each  $t \in [0, i]$, $a_t = b_t$.
	
	(c) For \( x = (a_0 1_0)(a_1 1_1) \cdots (a_i 1_i) \in \Gamma_k \), define \( \iota(x) = a_0 \) and \( l(x) = i \). $\iota(x)$ is the first natural integer that appears in $x$ and $l(x)$ is length of $x$.
	
	(d) For \( x, y \in \Gamma_k \), $ x $ and $y$ are irreducible if and only if  $l(x)\neq l(y)$  or \( \iota(x)\neq\iota(y) \). Otherwise, $x$ and $y$ are called compatible. 
	
	(e) Let $x_1,\ldots,x_m\in \Gamma_k$, we say that $\{x_1,\ldots,x_m\}$ is irreducible set if $x_i$ and $x_j$ are irreducible for every distinct $i,j\in[1,m]$. 
	
	(h) For \( x, y \in \Gamma_k \), \( x \prec y \) if and only if \( l(x) < l(y) \) and if \( l(x)=l(y) \), then \( \iota(x) < \iota(y) \). We say that $x\preceq y$ if and only if $x\prec y$ or $x=y$.
\end{definition} 
\begin{lemma}{\label{3.222}}
	$(\Gamma_k,\preceq)$ is totally ordered set.
\end{lemma}
\begin{proof}
The proof is routine.	
\end{proof}
\begin{lemma}{\label{{3.33}}}
	Let $\{x_1,\ldots,x_k\}$ be a finite subset of $\Gamma_k$. Then the following statement hold: \\
	(a) If $\{x_1,\ldots,x_k\}$ is irreducible set, then there exists a unique permutation $\sigma:[1,n]\to[1,n]$ such that 
	\[
	x_{\sigma(1)}\prec x_{\sigma(2)}\prec\cdots\prec x_{\sigma(n)}.
	\]
	(b) Let $	x_{1}\prec x_{2}\prec\cdots\prec x_{n}$. Then $\{x_1,\ldots,x_n\}$ is irreducible set.
\end{lemma}
\begin{proof}
	By Lemma \ref{3.222} and Definition $\prec$, it is obvious.
\end{proof}
\begin{definition}{\label{3.22}}
	For every $x = (a_0 1_0)(a_1 1_1) \cdots (a_i 1_i)$ and $ y = (b_0 1_0)(b_1 1_1) \cdots (b_j 1_j)$ in $\Gamma_k$, we define $x+y$ as follow:\\
	(a) If $x$ and $y$ are compatible, i.e., $\iota(x)=\iota(y)$ and $l(x)=l(y)$, define $x+y=(a_01_0)((a_1+b_1) 1_1) \cdots ((a_i+b_i) 1_i)$, and  \\
	(b) if $x$ and $y$ are irreducible we just write $x+y$. 
\end{definition}
In fact, the $"+"$ concatenates the two strings $x$ and $y$, and if the two strings are compatible, it assigns a simple string to them. We are now ready to define $"+"$ for a finite number of $\Gamma_k$ elements.
\begin{definition}{\label{3.55}}
	(a)	For every $n\in\mathbb{N}$, define 
	\[
	\Gamma_k^n=\left\{ x_1 + x_2 + \cdots + x_n : x_{1}\prec x_{2}\prec\cdots\prec x_{n}, \{x_1,\ldots,x_n\}\subseteq\Gamma_k\right\}.
	\]
	(b) Define $V_k =\cup_{i=1}^\infty \Gamma_k^n$, where $\Gamma_k^1=\Gamma_k$. $V_k$ is called \textbf{one variable symbolic polynomials space}. If \( \gamma = x_1 + x_2 + \cdots + x_n \) as in the definition of \( V_k \), then \( x_1, x_2, \ldots, x_n \) are the \textit{terms of \( \gamma \)}. The set of terms of $\gamma$ is denoted by $Term(\gamma)$.\\
	(c) Let \( \gamma = x_1 + x_2 + \cdots + x_n \) and \( \mu = y_1 + y_2 + \cdots + y_m \) be two elements in \( \Gamma^n_k \). We say that \( \gamma = \mu \) if and only if $Term(\gamma)=Term(\mu)$. 
\end{definition}
\begin{definition}{\label{3.366}}
	Given \( \gamma = x_1 + x_2 + \cdots + x_n \) and \( \mu = y_1 + y_2 + \cdots + y_m \) in \( V_k \) written as in the definition of \( V_k \), \( \gamma + \mu \) is defined as follows.
	\begin{enumerate}
		
		\item[(a)] 	Given \( t \in [1,n] \), there is at most one \( s \in [1,m] \) such that \( \iota(x_t) = \iota(y_s) \) and \( l(x_t) = l(y_s) \). 
		
		\item[(b)] For every \( t \in [1,n] \), if there is no such \( s \in \{1, 2, \ldots, m\} \) such that (a) holds. In fact, if  $\{x_1,\ldots,x_n\}\cup\{y_1,\ldots,y_m\}$ is a irreducible subset of $\Gamma_k$.
	\end{enumerate}
	(a) If there is such \( s \), assume that $x_t = (a_0 1_0)(a_1 1_1) \cdots (a_i 1_i)$, $ y_s = (b_0 1_0)(b_1 1_1) \cdots (b_i 1_i)$ such that \( b_0 = a_0 \). Let $z_t = (a_0 1_0)((a_1 + b_1) 1_1) \cdots ((a_i + b_i) 1_i)$. Having chosen \( z_1, z_2, \ldots, z_d \) for $0\leq d\leq min\{m,n\}$. By definition of $V_k$, $\{z_1,\ldots,z_d\}$ is irreducible subset of $\Gamma_k$. Now, let
	\[
	B = \left\{ y_s :\{x_1\ldots,x_n\}\cup\{y_s\}\text{ is irreducible subset of }\Gamma_k. \right\}.
	\]
	(Possibly \( B = \emptyset \).) 	Let \( q = |B| \) and let \( w_1, w_2, \ldots, w_{d+q} \) enumerate \( \{z_1, z_2, \ldots, z_d\} \cup B \) by $\prec$. Then we define
	\[
	\gamma + \mu = \mu + \gamma = w_1 + w_2 + \cdots + w_{d+q}.
	\]
	It is obvious that $\{w_1,w_2,\ldots,w_{d+q}\}$ is irreducible set.
	
	(b) If $\{x_1,\ldots,x_n\}\cup\{y_1,\ldots,y_m\}$ is a irreducible subset of $\Gamma_k$, let $w_1,\ldots,w_{n+m}$ enumerate $\{x_1,\ldots,x_n\}\cup\{y_1,\ldots,y_m\}$ by $\prec$. Then we define 
	\[
	\gamma+\mu=\mu+\gamma=w_1+w_2+\cdots+w_{n+m}.
	\]
\end{definition}
 
\begin{remark}{\label{3.44}}
	Let $x=x_1+\cdots+x_n$ and $y=y_1+\cdots+y_m$ be two elements of $V_k$. Then there exist finite subsets $I(x,y)$ of $Term(x)\cup Term(y)$ and $C(x,y)$ of $Term(x)\times Term(y)$ such that\\
(a) $C(x,y)=\{(a,b)\in Term(x)\times Term(y):l(a)=l(b),\iota(a)=\iota(b)\}$ and \\
(b) $I(x,y)=\{u\in Term(x)\cup Term(y): \{u\}\cup\{a+b:(a,b)\in C(x,y)\}\mbox{ is irreducible}\}$ is an irreducible subset of $Term(x)\cup Term(y)$. \\
	Then $E(x,y)=I(x,y)\cup \{a+b:(a,b)\in C(x,y)\}$ is irreducible, so let $z_1,\ldots,z_d$ enumerate $E(x,y)$ by $\prec$. Now define 
	\[
	x+y=z_1+\ldots+z_d.
	\]  
	Obviously, $I(x,y)=I(y,x)$ and $C(x,y)=C(y,x)$, and hence $E(x,y)=E(y,x)$. So $x+y=y+x$ for every $x,y\in V_k$.
\end{remark}
\begin{theorem}{\label{3.66}}
	Let $k\in\mathbb{N}$. Then $(V_k,+)$ is a commutative semigroup.
\end{theorem}
\begin{proof}
	By Definition \ref{3.366}, $V_k$ is closed under operation $+$, and by Remark \ref{3.44}, $(V_k,+)$ is commutative.
	
	Now we prove that $(x+y)+z=x+(y+z)$ for every $x,y,z\in V_k$, by induction on $|Term(z)|$.\\
	Let $z\in \Gamma_k$, and let $x,y\in V_k$ be two arbitrary elements. For $x,y\in V_k$, 
	\[
	Term(x+y)=A_x\cup A_y\cup\{a+b:(a,b)\in C(x,y)\}.
	\]
	Notice that $A_x=I(x,y)\cap Term(x)$, $A_y=I(x,y)\cap Term(y)$ and $\{a+b:(a,b)\in C(x,y)\}$ are disjoint. Let $Term(x+y)=\{u_1\prec u_2\prec\cdots\prec u_l\}$.\\
	{\bf{ Case 1:}}\\
	If there exists $u_i\in A_x$ such that $z$ and $u_i$ are compatible, then $(A_x\setminus\{u_i\})\cup\{u_i+z\}=A_{x,z}$ is irreducible, and also  $A_{x,z}\cup A_y\cup\{a+b:(a,b)\in C(x,y)\}$  and $Term(y)\cup\{z\}$  are irreducible. Therefore, we have 
	\[
	(x+y)+z=u_1+u_2+\cdots+(u_i+z)+\cdots+u_l=x+(y+z).
	\]
	{\bf{ Case 2:}}\\	
	If there exists $u_i\in A_y$ such that $z$ and $u_i$ are compatible, then $(A_y\setminus\{u_i\})\cup\{u_i+z\}=A_{y,z}$ is irreducible, and also  $A_{x}\cup A_{y,z}\cup\{a+b:(a,b)\in C(x,y)\}$  and $Term(y)\cup\{z\}$  are irreducible. Therefore we have 
	\[
	(x+y)+z=u_1+u_2+\cdots+(u_i+z)+\cdots+u_l=x+(y+z).
	\]
	{\bf{ Case 3:}}\\	
	If there exists $u_i\in \{a+b:(a,b)\in C(x,y)\}$ such that $z$ and $u_i$ are compatible, then there exist $a\in Term(x)$ and $b\in Term(y)$ such that $u_i=a+b$, $(a,z)\in C(x,z)$ and $(b,z)\in C(y,z)$.  Therefore we have 
	\[
	u_1\prec\cdots\prec u_{i-1}\prec u_i+z=a+b+z\prec u_{i+1}\prec \cdots\prec u_l.
	\]
	This implies that $(Term(y)\setminus{b})\cup\{b+z\}$ is irreducible. Therefore we have 
	\begin{align*}
		(x+y)+z&=u_1+u_2+\cdots u_{i-1}+((a+b)+z)+u_{i}\cdots+u_l\\
		&=u_1+u_2+\cdots u_{i-1}+(a+(b+z))+u_{i}\cdots+u_l\\
		&=x+(y+z).
	\end{align*}	
	{\bf{ Case 4:}}\\	
	If $z$ is not compatible with any member of $Term(x+y)$, then $Term(x+y)\cup\{z\}$ is irreducible. Therefore we will have
	\[
	u_1\prec u_2\prec\prec u_{i-1}\prec z\prec u_i\prec\cdots\prec u_l.
	\]
	Since $Term(y)\cup\{z\}$ is irreducible, implies that $(x+y)+z=x+(y+z)$.
	
	Now, assume that $|Term(z)|=n > 1$ and the statement is true for smaller sets, (induction hypothesis). Now let $z=z_1+\cdots+z_n$, let $x,y\in V_k$ be two arbitrary elements of $V_k$. Then, by the induction hypothesis,  we have	
	\begin{align*}
		(x+y)+z=&(x+y)+(z_1+z_2+\cdots+z_n)\\
		=&(x+y)+\left(z_1+(z_2+\cdots+z_n)\right)\\
		=&\left((x+y)+z_1\right)+(z_2+\cdots+z_n)\\
		=&\left(x+(y+z_1)\right)+(z_2+\cdots+z_n)\\
		=&x+\left((y+z_1)+(z_2+\cdots+z_n)\right)\\
		=&x+\left(y+(z_1+(z_2+\cdots+z_n))\right)\\
		=&x+\left(y+(z_1+z_2+\cdots+z_n)\right)\\
		&=x+(y+z).
	\end{align*}	
	Therefore, $"+"$ is an associative operation on $V_k$. 
	
\end{proof}

\begin{definition}
	For an element $\eta\in V_k$, we define 
	\[
	Ir(\{\eta\})=\{y\in V_k:Term(y)\cup Term(\eta)\mbox{ is an irreducible subset of }\Gamma_k\}.
	\] 
\end{definition}
\begin{lemma}{\label{366}}
	(a) Let $\eta\in V_k$, $Ir(\{\eta\})$ is a subsemigroup of $V_k$.\\ 
	(b) $\{Ir(\{\eta\}):\eta\in V_k\}$ has the finite intersection property.\\
	(c) For a finite subset $\{\eta_1,\ldots,\eta_m\}$ of $V_k$, $Ir(\{\eta_1,\ldots,\eta_m\})=\bigcap_{i=1}^mIr(\{\eta_i\})$ is a subsemigroup of $V_k$.
\end{lemma}
\begin{proof}
	(a) Let $L=\{\iota(v):v\in Term(\eta)\}$ then $\{x\in\Gamma_k:\iota(x)\notin L\}$ is a subset of $Ir(\{\eta\})$. Therefore $Ir(\{\eta\})$ is non-empty set. 
	
	Now, let $x,y\in Ir(\{\eta\})$. Then $Term(x+y)=\{u_1\prec u_2\prec\cdots\prec u_l\}$. Now, let $A=I(x,y)\cup \{a+b:(a,b)\in C(x,y)\}$ and let $A\cup Term(\eta)$ be not irreducible. So there exist $u\in A$ and $v\in Term(\eta)$ such that $u$ and $v$ are compatible. If $u\in I(x,y)$, we have a contradiction. So let $u=a+b$ for some $(a,b)\in C(x,y)$. Since $\iota(u)=\iota(a)=\iota(b)$ and $l(u)=l(a)=l(b)$, so $v$ and $a\in Term(x)$ are compatible. Therefore, we have a contradiction. Therefore $A\cup Term(\eta)$ is irreducible, and so $x+y\in Ir(\{\eta\})$. This completes our proof.\\
	\item[(b)] For every finite set $F \subseteq V_k$, let
$A = \left\{ \iota(x) \mid x \in \cup_{\eta \in F} \mathrm{Term}(\eta) \right\}.$ Now define 
	\[
	B=\{x\in\Gamma_k:\iota(x)\notin A\}.
	\] 
	It is obvious that $B\subseteq \bigcap_{\eta\in F}Ir(\{\eta\})$, and so $\{Ir(\{\eta\}):\eta\in V_k\}$ has the finite intersection property. \\
	(c) It is obvious.
\end{proof}
\begin{definition}{\label{3.6}}
	For $ r = ( r_1, \ldots, r_k)\in \mathbb{Z}^{k} $, $(a_0,a_1,\ldots,a_k) \in \mathbb{Z}^{k+1}$ and $a=(a_01)(a_11_1)(a_21_2)\cdots(a_i1_i)$, we define 
	\[
	r\bullet a=(a_01_0)(r_1a_11_1)(r_2a_21_2)\cdots(r_ia_i1_i). 
	\]
	Also, for $x,y\in \Gamma$, we define $\quad r \bullet (x+y) =r\bullet x+r\bullet y$.
\end{definition}
\begin{lemma}
	Let $k\in \mathbb{N}$, then the following statements hold:\\
	(a) For every $r\in \mathbb{Z}^k$ and $\eta\in \Gamma$, $r\bullet \eta$ is well defined. Also for every $\eta_1,\eta_2\in V_k$ and $r\in \mathbb{Z}^k$, we have $r\bullet(\eta_1+\eta_2)=r\bullet\eta_1+r\bullet\eta_2$.\\
	(b) For every  $a,b\in \mathbb{Z}^k$ and every $\eta\in V_k$, we have $(a+b)\bullet \eta=a\bullet\eta+b\bullet\eta$.
		\end{lemma}
\begin{proof}
	The proof is routine.
	\end{proof}

\begin{remark}
	Let $p(x)=\sum_{i=1}^ka_ix^i$ be a polynomial in $\mathbb{Z}[x]$.  Define $P_x:V_k\to \mathbb{Z}[x]$ by $P_x(a1_0)=a$ and $P_x(a1_i)=ax$ for every $i=1,2,\ldots,k$ and for every $a\in \mathbb{N}$. It is obvious that for every $u,v\in V_k$, if $u\in Ir(\{v\})$, then we have $P_x(u+v)=P_x(u)+P_x(v)$. Therefore
	\[
	P_x(\sum_{i=1}^k(a_i1_0)(1_1)(1_2)\cdots(1_i))=\sum_{i=1}^ka_ix^i.
	\]
	Therefore for every $p(x)\in\mathbb{Z}[x]$, there exists $\eta\in \V_k$ such that $P_x(\eta)=p(x)$.
\end{remark}

\subsection{Multivariable Symbolic Polynomial Space}

Pick \( k,m \in \mathbb{N} \) such that $k\geq 2$ and $m$ be sufficiently large number. For $i=1,2,\ldots,k$, define \( I_i=\{ 1_{i1}, \ldots, 1_{im}\} \), and $1_{A_{ij}}=(1_{i1})\cdots(1_{ij})=\Pi_{t=1}^j(1_{it})$ for every $A_{ij}=\{i1,i2,\ldots,ij\}\subseteq I_i$ and every $j=0,1,\ldots,m$. We assume that $A_{i0}=\emptyset$ for every $i=1,\ldots,k$. By Definition \ref{3.6}, recall
\[
a\bullet 1_{A_{ij}}=(a_11_{i1})\cdots(a_j1_{ij}),
\]
where $a=(a_1,\ldots,a_m)\in\mathbb{Z}^m$ and $j=1,\ldots m$. If $j=0$, then we define $a\bullet 1_\emptyset=1_\emptyset$.
\begin{definition} 
	(a) Let $k,m$ be two natural numbers. We define
	\[ 
	\Gamma(k,m) = \{ (a_0 1_0)(a_1\bullet 1_{A_{1j_1}})\cdots(a_k\bullet 1_{A_{kj_k}}):a_0\in\mathbb{Z}, a_i\in\mathbb{Z}^m,j_t\in [0,m],t\in[1,k]\}.
	\]
	(b) For \( x = (a_0 1_0)(a_1\bullet 1_{A_{1j_1}})\cdots(a_k\bullet 1_{A_{kj_k}}) \in \Gamma(k,m)\), define \( \iota(x) = a_0 \) and \( l(x) = (A_{1j_1},\ldots,A_{kj_k}) \). $\iota(x)$ is the first natural number that appears in $x$ and $l(x)$ is length of $x$.
	
	(c) If \( x = (a_0 1_0)(a_1\bullet 1_{A_{1j_1}})\cdots(a_k\bullet 1_{A_{kj_k}}) \) and \( y = (b_0 1_0)(b_1\bullet 1_{A_{1d_1}})\cdots(b_k\bullet 1_{A_{kd_k}}) \) are two members of \( \Gamma(k,m) \), then 
	$x = y $ if and only if $l(x)=l(y)$ and for each  $t \in \{0, 1, \ldots, k\}$ if $A_t\neq\emptyset$, $a_t = b_t$.
	
	(d) For \( x, y \in \Gamma(k,m) \), $ x $ and $y$ are irreducible if and only if  $l(x)\neq l(y)$  or \( \iota(x)\neq\iota(y) \). Otherwise, $x$ and $y$ are called compatible. 
	
	(e) Let $x_1,\ldots,x_m\in \Gamma(k,m)$, we say that $\{x_1,\ldots,x_m\}$ is irreducible set if $x_i$ and $x_j$ are irreducible for every distinct $i,j\in [1,m]$. 
	
%
\end{definition} 

\begin{definition}
	For every \( x = (a_0 1_0)(a_1\bullet 1_{A_{1j_1}})\cdots(a_k\bullet 1_{A_{kj_k}}) \) and \( y = (b_0 1_0)(b_1\bullet 1_{A_{1d_1}})\cdots(b_k\bullet 1_{A_{kd_k}}) \) in $\Gamma(k,m)$, we define $x+y$ as follow:\\
	(a) If $x$ and $y$ are compatible, i.e., $\iota(x)=\iota(y)$ and $l(x)=l(y)$, define $x+y=(a_0 1_0)((a_1+b_1)\bullet 1_{A_{1j_1}})\cdots((a_k+b_k)\bullet 1_{A_{kj_k}})$, and  \\
	(b) if $x$ and $y$ are irreducible we just write $x+y$. 
\end{definition}
The collection of all injective functions $\sigma:[1,n]\to [1,n]$ is denoted by $S(n)$. 
\begin{definition}
	(a) For every $n\in\mathbb{N}$, define 
	\[
	\Gamma^n=\{x_{\sigma(1)}+\cdots+x_{\sigma(n)}:\{x_1,\ldots,x_n\}\mbox{ is irreducible in }\Gamma(k,m)\mbox{ and }\sigma\in S(n)\}.
	\]
	(b) Let $\U(k,m)=\bigcup_{n=1}^\infty\Gamma^n$. If \( \gamma = x_1 + x_2 + \cdots + x_n \) as in the definition of \( \U(k,m) \), then \( x_1, x_2, \ldots, x_n \) are the \textit{terms of \( \gamma \)}. The set of terms of $\gamma$ is denoted by $Term(\gamma)$.\\
	(c) Let $\gamma,\mu\in\U(k,m)$. We say that $\gamma\sim\mu$ if and only if $Term(\gamma)=Term(\mu)$.
\end{definition}
\begin{lemma}
	Let $k,m\in\mathbb{N}$. Then $\sim$ is an equivalence relation on $\U(k,m)$.
\end{lemma}
\begin{proof}
	The proof is very simple.
\end{proof}
We denote equivalence class of an element $\gamma\in\U(k,m)$ by $[\gamma]$. Let $\gamma=x_{\sigma(1)}+\cdots+x_{\sigma(n)}$ for some $n\in\mathbb{N}$ and for some $\sigma\in S(n)$. For simplicity we assume that $[\gamma]=x_{1}+\cdots+x_{n}$. Therefore, we assume that 
\[
\frac{\U(k,m)}{\sim}=\{x_1+\cdots+x_n:\{x_1,\ldots,x_n\}\mbox{ is irreducible in }\Gamma(k,m)\mbox{ and }n\in\mathbb{N}\}.
\]

We define 
\[
\V(k,m)=\frac{\U(k,m)}{\sim}.
\]
\begin{definition}
	Let \( \gamma = x_1 + x_2 + \cdots + x_n \) and \( \mu = y_1 + y_2 + \cdots + y_m \) be two elements in \( V(k,m) \). We say that \( \gamma = \mu \) if and only if $Term(\gamma)=Term(\mu)$. 
\end{definition}
{\bf{ Definition of operation $"+"$ on $\V(k,m)$ :}} Given \( \gamma = x_1 + x_2 + \cdots + x_u \) and \( \mu = y_1 + y_2 + \cdots + y_v \) in \( V(k,m) \) written as in the definition of \( V(k,m) \), \( \gamma + \mu \) is defined as follows.
\begin{enumerate}
	
	\item[(a)] 	Given \( t \in [1,u] \), there is at most one \( s \in [1,v] \) such that \( \iota(x_t) = \iota(y_s) \) and \( l(x_t) = l(y_s) \), i.e. for some $x_t\in Term(\gamma)$, there exists $y_s\in Term(\mu)$ such that $x_t$ and $y_s$ are compatible.
	
	\item[(b)] For every \( t \in [1,u] \), if there is no such \( s \in [1,v] \) such that (a) holds. In fact, if  $\{x_1,\ldots,x_u\}\cup\{y_1,\ldots,y_v\}$ is a irreducible subset of $\Gamma(k,m)$.
\end{enumerate}
(a) If there is such \( s \), assume that $x_t = (a_0 1_0)(a_1\bullet 1_{A_1})\cdots(a_k\bullet 1_{A_k})$, $ y_s = (b_0 1_0)(b_1\bullet 1_{A_1})\cdots(b_k\bullet 1_{A_k})$ such that \( b_0 = a_0 \). Let $z_t = (a_0 1_0)((a_1 + b_1) 1_{A_1}) \cdots ((a_k + b_k) 1_{A_k})$. Having chosen \( z_1, z_2, \ldots, z_d \) for $0\leq d\leq min\{u,v\}$. By definition of $\V(k,m)$, $\{z_1,\ldots,z_d\}$ is irreducible subset of $\Gamma(k,m)$. Now, let
\[
B = \left\{ y_s :\{x_1\ldots,x_n\}\cup\{y_s\}\text{ is irreducible subset of }\Gamma(k,m). \right\}.
\]
(Possibly \( B = \emptyset \).) 	Let \( q = |B| \) and let \( w_1, w_2, \ldots, w_{d+q} \) enumerate \( \{z_1, z_2, \ldots, z_d\} \cup B \). Then we define
\[
\gamma + \mu = \mu + \gamma = w_1 + w_2 + \cdots + w_{d+q}.
\]
It is obvious that $\{w_1,w_2,\ldots,w_{d+q}\}$ is irreducible set.

(b) If $\{x_1,\ldots,x_n\}\cup\{y_1,\ldots,y_m\}$ is a irreducible subset of $\Gamma_k$, let $w_1,\ldots,w_{n+m}$ enumerate $\{x_1,\ldots,x_n\}\cup\{y_1,\ldots,y_m\}$. Then we define 
\[
\gamma+\mu=\mu+\gamma=w_1+w_2+\cdots+w_{n+m}.
\]
\begin{theorem}
	Let $k,m\in\mathbb{N}$ and $k\geq 2$. Then $(V(k,m),+)$ is commutative semigroup.
\end{theorem}
\begin{proof}
	The proof is similar to the proof of Theorem \ref{3.66}.	
\end{proof}

\begin{definition}
	For an element $\eta\in V(k,m)$, we define 
	\[
	Ir(\{\eta\})=\{y\in V(k,m):Term(y)\cup Term(\eta)\mbox{ is an irreducible subset of }\Gamma(k,m)\}.
	\] 
\end{definition}
\begin{lemma}
	(a) Let $\eta\in V(k,m)$, $Ir(\{\eta\})$ is a subsemigroup of $V(k,m)$.\\ 
	(b) $\{Ir(\{\eta\}):\eta\in V(k,m)\}$ has the finite intersection property.\\
	(c) For a finite subset $\{\eta_1,\ldots,\eta_m\}$ of $V(k,m)$, $Ir(\{\eta_1,\ldots,\eta_m\})=\bigcap_{i=1}^mIr(\{\eta_i\})$ is a subsemigroup of $V(k,m)$.
\end{lemma}
\begin{proof}
	See Lemma \ref{366}.
\end{proof}
\begin{remark}
Let \(p(x_1, x_2, \ldots, x_k)=\sum_{i=1}^d a_{i} x_1^{\alpha_{1i}}x_2^{\alpha_{2i}}\cdots x_k^{\alpha_{ki}}\) be an element of $\mathbb{Z}([x_1,\ldots,x_k])$. Let $m=\sum_{i=1}^d\sum_{t=1}^k\alpha_{ti}$. Define $\eta$ in $\Gamma(k,m)$ as follow: 
\[
\eta=\sum_{i=1}^d(a_i1_0)(1_{A1\alpha_{1i}})(1_{A_{2\alpha_{2i}}})\cdots(1_{A_{k\alpha_{ki}}}), 
\]
where $A_{j\alpha_{ji}}=\{1_{j1},1_{j2}\ldots,1_{j\alpha_{ji}}\}\subseteq I_j$ and $|A_{j\alpha_{ji}}|=\alpha_{ji}$. 

Now define 
$P_{x_1,\ldots,x_k}:V(k,m)\to\mathbb{Z}[x_1,\ldots,x_k]$ by $P_{x_1,\ldots,x_k}(1_{A_{j\alpha_{ji}}})=x_j^{|A_{j\alpha_{ji}}|}=x_j^{\alpha_{ji}}$, $P_{x_1,\ldots,x_k}(a1_0)=a$ and $P_{x_1,\ldots,x_k}((1_A)(1_B))=P_{x_1,\ldots,x_k}(1_A)P_{x_1,\ldots,x_k}(1_B)$ for every $a\in\mathbb{Z}$ and $1_A,1_B\in \Gamma(k,m)$. Since 
$P_{x_1,\ldots,x_k}(\eta_1+\eta_2)=P_{x_1,\ldots,x_k}(\eta_1)+P_{x_1,\ldots,x_k}(\eta_2)$ for every irreducible elements $\eta_1,\eta_2\in V(k,m)$, we will have 
\[
P_{x_1,\ldots,x_k}(\eta)=\sum_{i=1}^d a_{i} x_1^{\alpha_{1i}}x_2^{\alpha_{2i}}\cdots x_k^{\alpha_{ki}}=p(x_1, x_2, \ldots, x_k).
\]  
Therefore, for every multivariable polynomial 
\[
p(x_1, x_2, \ldots, x_k)=\sum_{i=1}^d a_{i} x_1^{\alpha_{1i}}x_2^{\alpha_{2i}}\cdots x_k^{\alpha_{ki}},
\]
 there exist $k,m\in\mathbb{N}$ such that for some $\eta\in V(k,m)$ we will have be a multivariable polynomial with non-negative coefficients
\end{remark}
\section{Main Results}
We are now ready to take the fundamental steps toward proving the polynomial version of van der Waerden's theorem. To this end, we first show that for any finite collection of symbolic polynomials, there exist monochromatic structures of the form \( x + f \bullet \eta \).\\
For a nonempty set $J$ and $k\in\mathbb{N}$, we define $\Delta^k_J=\{(s,s,\ldots,s)\in \times_{i=1}^kJ:s\in J\}$.
\begin{definition}
	Let $\{\eta_1,\ldots,\eta_m\}$ be an arbitrary non-empty finite subset of $V_k$.\\
	(a) We define $\Ir=\Ir(\{\eta_i\}_{i=1}^m)=\bigcap_{i=1}^m Ir(\{\eta_i\})$.\\
	(b) Define 
	\[
	\I=I(\{\eta_i\}_{i=1}^m)=\left\{ \left(x + r \bullet \eta_1, \ldots, x  + r \bullet \eta_m \right) : x \in \Ir,\,r \in \Delta_{\mathbb{Z}}^k\right\}.
	\]
	(c) Define $\V=V(\{\eta_i\}_{i=1}^m) =\I \cup \Delta_{\Ir}^m$.\\
	(d) For every $i\in[1,m]$, let
	\[
	S_i=\{x+r\bullet \eta_i:r\in\Delta_{\mathbb{Z}}^k, x\in \Ir\}\cup\Ir, 
	\]
	and define $S=S(\{\eta_i\}_{i=1}^m)=\bigcup_{i=1}^mS_i$.
\end{definition}
Clearly, $(S_i,+)$ is a subsemigroup of $(V_k,+)$  and \((\Ir,+)\) is a subsemigroup of \((S_i,+)\) for every $i=1,\ldots,m$. For every $x,y\in S$, we define operation $"\dot{+}"$ on $S$ as the following way:
\[
x\dot{+}y=x+y\quad\mbox{if and only if }\quad\exists\, i\in[1,m]\quad x,y\in S_i.
\]
Since $\V$ is a subset of $\times_{i=1}^mS_i$, so we could define operation $\dot{+}$ on $\V$ as follow: 
\[
(x_1,\dots,x_k)\dot{+}(y_1,\ldots,y_k)=(x_1\dot{+}y_1,\dots,x_k\dot{+}y_k),
\]
for every $(x_1,\dots,x_k),(y_1,\ldots,y_k)\in \V$.

\begin{lemma}{\label{asli}}
	Let $\{\eta_1,\ldots,\eta_m\}$ be an arbitrary non-empty finite subset of $V_k$. The following statements hold:\\
	(a) Let $S=S(\{\eta_i\}_{i=1}^m)$. Then $(S,\dot{+})$ is an adequate commutative partial semigroup.\\
	(b) Let $\V=V(\{\eta_i\}_{i=1}^m)$. Then $(\V,\dot{+})$ is an  adequate commutative partial semigroup.\\
	(c) Let $\I=(I(\{\eta_i\}_{i=1}^m),\dot{+})$. Then $(\I,\dot{+})$ is an adequate partial subsemigroup of $(\V,\dot{+})$ and $K(\delta \V)=K(\delta\I)$, where $\delta \V=\bigcap_{u\in \V}cl_{\times_{i=1}^m\beta S}R_\V(u)$ and $\delta \I=\bigcap_{u\in \V}cl_{\times_{i=1}^m\beta S}R_\I(u)$. \\
	(d) $(\Ir,\dot{+})$ is an adequate partial subsemigroup of $(S,\dot{+})$ and $\delta S=\beta \Ir$.\\
	(e) Let $Y=\times_{i=1}^m\delta S$. Then $K(\delta V)=K(Y)\cap \delta V$.
\end{lemma}
\begin{proof}
	(a) It is obvious. \\	
	(b) It is obvious that $(\V,\dot{+})$ is commutative partial semigroup. Since for every non-empty finite subset $\{u^1,\ldots,u^l\}$ of $\V$, we have 
	\begin{align*}
	R_\V(\{u^1,\ldots,u^l\})&=\{u\in\V:u\dot{+}u^j\mbox{ is well define for every }j=1,2,\ldots,l\}\\
	&=\times_{i=1}^m(\cap_{j=1}^l R_{S_i}(u^j_i))\\
	&\supset\times_{i=1}^m\Ir,
	\end{align*}
	where $u^j=(u_1^j,\ldots,u_m^j)$ for every $j=1,\ldots, l$. Therefore $(\V,\dot{+})$ is adequate.
	
	(c) Similarly to part (b), has been proved $(\I,\dot{+})$ is an adequate partial semigroup. Now, $u=(u_1,\ldots,u_m)\in \V$, we have 
	\begin{align*}
	R_\V(u_1,\ldots,u_m)=&\times_{i=1}^mR_{S_i}(u_i)\\
	&=\times_{i=1}^mS_i\\
	&\supseteq\times_{i=1}^m\Ir\\
	&=\times_{i=1}^mR_\I(\{\eta_1\ldots,\eta_m\}).
	\end{align*}
	Therefore by Definition \ref{1.11}, $(\I,\dot{+})$ is an adequate partial subsemigroup of $(\V,\dot{+})$. It is obvious that $\I$ is ideal of $\V$, and so by Theorem \ref{1.23}, $K(\delta \I)=K(\delta \V)$.
	
	(d) The first part is obvious. Notice that for every $i=1,\ldots,m$, $S_i=R_i\cup\Ir$, where $R_i=\{x+r\bullet \eta_i:r\in\Delta_{\mathbb{Z}}^k, x\in \Ir\}$. Since $R_i\cap R_j=\emptyset$ for every distinct $i,j=1,\ldots,m$,  Therefore 
	\begin{align*}
	\delta S=&\bigcap_{x\in S}\overline{R_S(x)}\\
	=&\bigcap_{i=1}^m(\bigcap_{x\in S_i}\overline{R_S(x)})\\
	=&\bigcap_{i=1}^m(\bigcap_{x\in S_i}\overline{ S_i})\\
		=&\bigcap_{i=1}^m(\overline{\Ir}\cup\overline{R_i})\\
		=&\overline{\Ir}\cup(\bigcap_{i=1}^m\overline{R_i})\\
		=&\beta \Ir.
	\end{align*}
	
	(e) Let $Y=\times_{i=1}^m\delta S$. Let \( p \in K(\beta\Ir)=K(\delta S)\), so \(\overline{p} = (p, p, \ldots, p) \in K(Y) \). We claim $\overline{p}\in \delta \V$. Therefore let $U$ be a neighborhood of $\overline{p}$ and let $x\in \V$, so there exists $C_1,\ldots, C_m\in p$ such that $\times_{t=1}^mC_i\subseteq U$. Now pick $a\in\bigcap_{i=1}^mC_i$ and so $\overline{a}=(a,\ldots,a)\in U\cap(\times_{i=1}^m\Ir)\subseteq U\cap R_\V(x)$. This implies that $ \overline{p} \in K(Y)\cap \delta V$ and so $K(\delta \V)=K(Y)\cap \delta \V$. Now, by Theorem 1.65 in \cite{Hindman} implies that $K(\delta \V)=K(Y)\cap\delta \V$, and so $K(\delta \V)\subseteq \delta \I$. 
\end{proof}

\begin{theorem}{\label{5.9}}
	Let $\{\eta_1,\ldots,\eta_m\}$ be an arbitrary non-empty finite subset of $V_k$, let $\Ir=\bigcap_{i=1}^m Ir(\{\eta_i\})$ and let $S=S(\{\eta_i\}_{i=1}^m)$. Let  $A\subseteq (S,\dot{+})$ be a partially piecewise syndetic set. Then there exist \( a \in \Ir \) and \( r\in \Delta_{\mathbb{Z}}^k \) such that  
	\[
	\{ a + r\bullet \eta_1, \ldots, a + r\bullet \eta_m \}\subseteq A.
	\] 
\end{theorem}
\begin{proof}
	
	By using \(\eta_1, \ldots, \eta_m\) and \(\Ir\), we define    
	\[
	\V = \left\{ \left(x\dot{ +} r \bullet \eta_1, \ldots, x \dot{+} r \bullet \eta_m \right) : x \in \Ir,\,r \in \Delta_{\mathbb{Z}}^k \right\}\bigcup \Delta_\Ir^m.
	\]
	Also, let $\I=\left\{ \left(x + r \bullet \eta_1, \ldots, x + r \bullet \eta_m \right) : x \in S,\,r \in \Delta_{\mathbb{Z}}^k\right\}$. Then, $(\V, \dot{+})$ is a commutative adequate partial semigroup.
	
	Since $A\subseteq S$ is partially piecewise syndetic, so $\overline{A}\cap K(\delta S)\neq\emptyset$. By Lemma \ref{asli}(d), implies that  \( p \in K(\beta\Ir)\cap \overline{A}\) exists, so \(\overline{ p} = (p, p, \ldots, p) \in K(\delta \V)=K(\times_{i=1}^m\beta S_i)\cap\delta \V \) and \( \overline{p} \in K(\V) \subseteq \delta \I \). Since $A\in p$, then \( \I \cap \times_{t=1}^{m} A \neq \emptyset \). Pick \( z \in \I \cap \times_{t=1}^{m} A \), and choose \( a \in \Ir \) and \( r\in \Delta_{\mathbb{Z}}^k \) such that $z = \left( a + r\bullet \eta_1, \dots, a + r\bullet \eta_m \right)$. Therefore  there exist \( a \in \Ir \) and \( r\in \Delta_{\mathbb{Z}}^k \) such that
	\[
	\{a + r\bullet \eta_1, \dots, a + r\bullet \eta_m \}\subseteq A.
	\]
\end{proof}

\begin{definition}{\label{pi}}
Define function $\pi:V_k\to \mathbb{Z}$ with the following properties:\\
	(a) \(\pi(1_i) = 1 \quad \forall i=0,\ldots,k\).\\
	(b)	\(\pi(a1_i) = a \quad \forall a\in \mathbb{Z}, \quad \forall i=0,1,\ldots,k\).\\
(c)	\(\pi(\prod_{j=0}^i (a_j 1_j)) = \prod_{j=0}^i \pi(a_j 1_j)=\prod_{j=0}^i a_j \quad \forall \prod_{j=0}^i (a_j 1_j)\in \Gamma_k\).\\
(d) $\pi(a+b) = \pi(a) + \pi(b)$ if $a\in Ir({b})$.
\end{definition}
	Therefore if $a_0=b_0$, we will have
	\[
	\pi(\prod_{j=0}^i (a_j 1_j)+\prod_{j=0}^i (b_j 1_j))=\pi((a_01_0)\prod_{j=1}^i ((a_j+b_j) 1_j))=(a_0)\prod_{j=1}^i (a_j+b_j).
	\]

\begin{theorem}
	Let $p_1,\ldots,p_m\in x\mathbb{Z}[x]$ be polynomials for $i=1,\ldots m$‎. ‎Then for any finite partition $\mathcal{C}$ of \( \mathbb{Z} \), ‎there exist a cell $C\in\mathcal{C}$ and $d,r\in \mathbb{Z}$ such that
	\[
	\{d+p_i(r):i=1,2,\ldots,m\}\subseteq C.
	\]‎ 
\end{theorem}‎
‎\begin{proof}‎
	Let $p_j(x)=\sum_{i=1}^{k_j}a_{ij}x^{i}$ for $j=1,\ldots,m$ and let $k=\max\{k_1,\ldots,k_m\}$‎. ‎Define $\eta_j=\sum_{i=1}^{k_j}a_{ij}(1_0)(1_1)\cdots(1_i)\in V_k$ for $j=1,\ldots,m$‎. ‎Now, assume that $\mathcal{C}$ is a finite partition of $\mathbb{Z}$. Then $\{\pi^{-1}(C):C\in\mathcal{C}\}$ is a finite partition for $S(\{\eta_1,\ldots,\eta_m\})$‎, and so, one of the cells of $\{\pi^{-1}(C):C\in\mathcal{C}\}$ is piecewise syndetic. ‎Now by Theorem \ref{5.9}‎, there exist \( a \in \Ir \) and \( \overline{r}=(r,\ldots,r)\in \Delta_{\mathbb{Z}}^k \) such that
	\[
	\{a + \overline{r}\bullet \eta_1, \dots, a + \overline{r}\bullet \eta_m \}\subseteq \pi^{-1}(C).
	\]
	Now, by Definition \ref{pi}, we will have
	‎\begin{align*}‎	
	‎	\pi(a + \overline{r}\bullet \eta_j)=&\pi(a)+\pi(\overline{r}\bullet\sum_{i=1}^{k_j}((c_i1)(1_1)(1_2)\cdots(1_i)))\\‎
	‎	=&\pi(a)+\pi(\sum_{i=1}^{k_j}((c_i1)(r1_1)(r1_2)\cdots(r1_i)))\\‎
	=&\pi(a)+p_j(r)‎.
	\end{align*}‎
	Therefore $\{p_i(r):i=1,2,\ldots,m\}\subseteq C$‎. Now, let $d=\pi(a)$, and this completes our proof.
\end{proof}
\begin{remark}{\label{4.77}}
(a) ‎Let $F\in P_f(\mathbb{N})$‎, ‎$f:\mathbb{N}\to \mathbb{Z}$ be a sequence, ‎$\eta = \sum_{i=1}^k (c_i1)(1_1)(1_2)\cdots(1_i)$ be a one variable polynomial, let $k\in\mathbb{N}$ and let $\overline{f(t)}=(f(t),\ldots,f(t))$ be a $k-$tuple vector in  $\mathbb{Z}^k$. ‎Now, ‎we define $T_F^{\eta}f:V_k\to V_k$ by
‎\begin{align*}‎
‎T_F^{\eta}f(x)=&x+\sum_{t\in F} \overline{f(t)}\bullet\eta\\
=&x+\sum_{t\in F}\left(\sum_{i=1}^k (c_i1)(f(t)1_1)(f(t)1_2)\cdots(f(t)1_i)\right)\\‎
‎=&x+\sum_{i=1}^k\left(((c_i1)((\sum_{t\in F}f(t))1_1)((\sum_{t\in F}f(t))1_2)\cdots((\sum_{t\in F}f(t))1_i))\right)‎,
‎\end{align*}‎
‎where $x\in V_k$‎. \\
(b) For every $\eta\in\V_k$ and for every $F,G\in P_f(\mathbb{N})$ , $T_F^\eta f\circ T_G^\eta f(x)=T_{F\cup G}^\eta f(x)$ if $F\cap G=\emptyset$.
\end{remark}
\begin{remark}
	Let $\{\eta_1,\ldots,\eta_m\}$ be an arbitrary non-empty finite subset of $V_k$, let $\Ir=\bigcap_{i=1}^m Ir(\{\eta_i\})$ and Let $f:\mathbb{N}\to\mathbb{Z}$ be a sequence.
	
(a)	 For every $i=1,\ldots,m$, define 
	\[
	S_i=\{T_F^{\eta_i}f(x):F\in P_f(\mathbb{N}), x\in \Ir\}\cup \Ir.
	\]
	 For every $i\in\{1,\ldots,m\}$, we define $\dot{+}_1$ on $S_i$ by $\dot{+}_1(x,y)=x+y$ if $x$ and $y$ are irreducible or there exist disjoint finite subsets $F$ and $G$ in $P_f(\mathbb{N})$ such that  $x=T_F^{\eta_i}f(x_1)$ and $T_G^{\eta_i}f(x_2)$ for some $x_1,x_2\in \Ir$, and so we will have  	 
	\[
x\dot{+}_1y=	T_F^{\eta_i}f(x_1)\dot{+}_1T_G^{\eta_i}f(x_2)=T^{\eta_i}_{F\cup G}f(x_1+x_2).
	\]
	It is obvious that, $(S_i,\dot{+}_1)$ is commutative adequate partial semigroup and \(\Ir\) is a subsemigroup of \(S_i\) for every $i=1,\ldots,m$.
	
(b) Define $S=S(\{\eta_1,\ldots,\eta_m\})=\bigcup_{i=1}^mS_i$. For every $x,y\in S$, we define operation $"\ddot{+}"$ on $S$ as the following way:
\[
x\ddot{+}y=x\dot{+}_1y\quad\mbox{if and only if }\exists\,i\in\{1,\ldots,m\}\quad x,y\in S_i.
\]
It is obvious that $(S,\ddot{+})$ is a commutative adequate partial semigroup. In fact, for every $x\in S_i$, $R_{S_i}(x)\supseteq\Ir$ and $R_S(x)=S_i$ if $x\in \Ir$, and so $\delta S=\bigcap_{x\in S}\overline{R_S(x)}=\beta\Ir$, see Lemma \ref{asli}(d).

(c) Define 
$$I=\left\{ \left(T_F^{\eta_1}f(x), \ldots, T_F^{\eta_m}f(x)\right) : x \in \Ir,\,F \in P_f(\mathbb{N}) \right\}.$$

(d) Define $V=I\cup \Delta_\Ir^k$ and consider the operator $\ddot{+}$ as component-wise addition on $V$, i.e., for two elements $u = (u_1, u_2, \dots, u_m)$ and $v = (v_1, v_2, \dots, v_m)$ in $\times_{i=1}^mS$, we have
\[
u \ddot{+} v = (u_1\ddot{+} v_1, u_2 \ddot{+} v_2, \dots, u_m \ddot{+} v_m).
\]
Then, $(V, \ddot{+})$ is a commutative adequate partial semigroup. It is obvious that $R_V(u_1,\ldots,u_m)=\times_{i=1}^mR_S(u_i)$ and we define 
\[
\delta V=\bigcap_{(u_1,\ldots,u_m)\in V}cl_{\times_{i=1}^m\beta S}R_S(u_i).
\]
(e) It is obvious that $I$ is an adequate partial semigroup under operation $"\ddot{+}"$, and for every $(u_1,\ldots,u_m)\in V$, we have 
\begin{align*}
R_V(u_1,\ldots,u_m)=&\times_{i=1}^mR_{S_i}(u_i)\\
\supseteq&\times_{i=1}^m\Ir\\
&=\times_{i=1}^mR_I(\{\eta_1\ldots,\eta_m\}).
\end{align*}
Therefore by Definition \ref{1.11}, $(I,\ddot{+})$ is an adequate partial subsemigroup of $(V,\ddot{+})$ and so by Lemma \ref{1.10}, $\delta I\subseteq \delta V$.

(f) It is obvious that $\I$ is ideal of $\V$, and so by Theorem \ref{1.23}, $K(\delta \I)=K(\delta \V)$.(See Lemma \ref{asli}).

(h) $(\Ir,\dot{+})$ is an adequate partial subsemigrpup of $(S,\dot{+})$ and $\delta S=\beta \Ir$. Notice that for every $i=1,\ldots,m$, $S_i=R_i\cup\Ir$, where 
\[
R_i=\{x+\sum_{t\in F}\overline{f(t)}\bullet \eta_i:F\in P_f(\mathbb{N}), x\in \Ir\}.
\]
 Since $R_i\cap R_j=\emptyset$ for every distinct $i,j=1,\ldots,m$, so  
\begin{align*}
\delta S=&\bigcap_{x\in S}\overline{R_S(x)}\\
=&\bigcap_{i=1}^m(\bigcap_{x\in S_i}\overline{R_S(x)})\\
=&\bigcap_{i=1}^m(\bigcap_{x\in S_i}\overline{(R_S(x)\cap S_i)})\\
=&\bigcap_{i=1}^m(\bigcap_{x\in S_i}\overline{(R_S(x)\cap \Ir)}\cup\overline{(R_S(x)\cap R_i)})\\
=&\bigcap_{i=1}^m\bigcap_{x\in S_i}(\overline{\Ir}\cup\overline{(R_S(x)\cap R_i)})\\
=&\bigcap_{i=1}^m\overline{\Ir}\cup(\bigcap_{x\in S_i}\overline{(R_S(x)\cap R_i)})\\
=&\overline{\Ir}\cup\bigcap_{i=1}^m(\bigcap_{x\in S_i}\overline{(R_S(x)\cap R_i)})\\
=&\overline{\Ir}\qquad\qquad\mbox{ because }\bigcap_{i=1}^m(\bigcap_{x\in S_i}\overline{(R_S(x)\cap R_i)})=\emptyset\\
=&\beta \Ir\\
\end{align*}

(k) Let $Y=\times_{i=1}^m\delta S$, then  $K(\delta V)=K(Y)\cap \delta V$.(See Lemma \ref{asli}).

\end{remark}
\begin{theorem}{\label{5.9}}‎
	Let $\{\eta_1,\ldots,\eta_m\}$ be an arbitrary non-empty finite subset of $V_k$, let $\Ir=\bigcap_{i=1}^m Ir(\{\eta_i\})$ and let $f:\mathbb{N}\to \mathbb{Z}$ be a sequence. Let  $S=S(\{\eta_1,\ldots,\eta_m\})=\bigcup_{i=1}^mS_i$ and define $‎T_F^{\eta_i}f(x)= x+\sum_{t\in F} \overline{f(t)}\bullet\eta_i$ for every $F\in P_f(\mathbb{N})$, $x\in \Ir$ and $i\in\{1,\ldots,m\}$. Let $A\subseteq S$ be a partially piecewise syndetic. Then there exist \( x \in \Ir \) and \( F\in P_f(\mathbb{N})\) such that 
	\[
	‎\{T_F^{\eta_1}f(x),T_F^{\eta_2}f(x),\ldots,T_F^{\eta_m}f(x)\}\subseteq A‎.
	\]
\end{theorem}‎
\begin{proof}
	Pick some \( p \in K(\delta S)\cap \overline{A}\), so \(\overline{ p} = (p, p, \ldots, p) \in K(\delta V)=K(\times_{i=1}^m\delta S_i) \) and \( \overline{p} \in K(V) \subseteq \delta I \). Now if $A\in p$, then \( I \cap \times_{t=1}^{m} A \neq \emptyset \). Pick \( z \in I \cap \times_{t=1}^{m} A \), and choose \( a \in \Ir \) and \( F\in P_f(\mathbb{N}) \) such that  
	\[
	z = \left(T_F^{\eta_1}(a), \ldots, T_F^{\eta_m}(a)\right)\in\times_{i=1}^mA.
	\] 
	This completes our proof.
\end{proof}

\begin{theorem}[Van der Waerden Polynomial version]{\label{van}}
Let $p_1,\ldots,p_m\in \mathbb{Z}[x]$ be polynomials such that $p_i(0)=0$ for $i=1,\ldots m$. Then for any finite partition $\mathcal{C}$ of \( \mathbb{Z} \) and every sequence $f:\mathbb{N}\to \mathbb{Z}$, there exist a cell $C\in\mathcal{C}$,  \(a \in \mathbb{Z} \) and $F\in P_f(\mathbb{N})$ such that
	 \[
	\{a+p_i(\sum_{t\in F}f(t)):i=1,2,\ldots,m\}\subseteq C.
	\] 
\end{theorem}
\begin{proof}
Let \( p_j(x) = \sum_{i=1}^{k_j} a_{ij} x^i \) for each \( j = 1, \ldots, m \), and let \( k = \max\{k_1, \ldots, k_m\} \). Define 
$\eta_j = \sum_{i=1}^{k_j} a_{ij}(1_0)(1_1)\cdots(1_i) \in V_k,$
and pick \( f : \mathbb{N} \to \mathbb{Z} \). Let \( \mathcal{C} \) be a finite partition of \( \mathbb{Z} \). Then \( \{ \pi^{-1}(C) : C \in \mathcal{C} \} \) form a finite partition of the symbolic polynomial space \( V_k \). By Theorem \ref{5.9}, there exist \( x \in \Ir = \bigcap_{j=1}^m \operatorname{Ir}(\{\eta_j\}) \), a finite set \( F \subseteq \mathbb{N} \),  for some \( C \in \mathcal{C} \), implies that 
\[
\{ T_F^{\eta_1} f(x), T_F^{\eta_2} f(x), \ldots, T_F^{\eta_m} f(x) \} \subseteq \pi^{-1}(C).
\]
This implies that
\[
\{ \pi(T_F^{\eta_1} f(x)), \pi(T_F^{\eta_2} f(x)), \ldots, \pi(T_F^{\eta_m} f(x)) \} \subseteq C.
\]

By Definition \ref{pi} and Remark \ref{4.77}, we have
\[
T_F^{\eta_j} f(x) = x + \sum_{i=1}^{k_j} a_{ij} \cdot (1_0) \cdot \left( \sum_{t \in F} f(t) \cdot 1_1 \right) \cdots \left( \sum_{t \in F} f(t) \cdot 1_i \right).
\]
and so 
\begin{align*}
	\pi(T_F^{\eta_j} f(x)) 
	&= \pi\left( x + \sum_{i=1}^{k_j} a_{ij} (1_0) \prod_{\ell=1}^{i} \left( \sum_{t \in F} f(t) \cdot 1_\ell \right) \right) \\
	&= \pi(x) + \sum_{i=1}^{k_j} a_{ij} \cdot \pi\left( (1_0) \prod_{\ell=1}^{i} \left( \sum_{t \in F} f(t) \cdot 1_\ell \right) \right).
\end{align*}

Now, using the properties of \( \pi \), we have \( \pi((1_0)) = 1 \), and \( \pi(\sum_{t \in F} f(t) \cdot 1_\ell) = \sum_{t \in F} f(t) \), since \( \pi(a \cdot 1_\ell) = a \). Thus,
\[
\pi\left( (1_0) \prod_{\ell=1}^{i} \left( \sum_{t \in F} f(t) \cdot 1_\ell \right) \right) 
= 1 \cdot \prod_{\ell=1}^{i} \sum_{t \in F} f(t) 
= \left( \sum_{t \in F} f(t) \right)^i.
\]
Substituting this into the previous relation, we have:
\[
\pi(T_F^{\eta_j} f(x)) = \pi(x) + \sum_{i=1}^{k_j} a_{ij} \left( \sum_{t \in F} f(t) \right)^i = \pi(x) + p_j\left( \sum_{t \in F} f(t) \right).
\]

Now, assume that $a=\pi(x)$. This implies that 
\[
\{ a + p_j(\sum_{t \in F} f(t)) : j = 1, \ldots, m \} \subseteq C,
\]
which completes the proof.
\end{proof}
We are now ready to state Van der Waerden's theorem for multivariate polynomials.\\
Assume that $k,m\in\mathbb{N}$ and let 
\[
\eta=\sum_{i=1}^d(a_i1_0)(1_{A1\alpha_{1i}})(1_{A_{2\alpha_{2i}}})\cdots(1_{A_{k\alpha_{ki}}}), 
\]
be an element of $V(k,m)$, see subsection 3.2. Define $\overline{f(t)}=(f(t),\ldots,f(t))\in\mathbb{Z}^m$ for every sequence $f:\mathbb{N}\to\mathbb{Z}$. Let $f_1,f_2,\ldots,f_k:\mathbb{N}\to \mathbb{Z}$ be arbitrary sequences.  For $F=\times _{i=1}^kF_i\in \times_{i=1}^kP_f(\mathbb{N})$, we have 
	\begin{align*}
	\sum_{(t_1,\ldots,t_k)\in F}	(\overline{f_1(t_1)},\ldots,\overline{f_k(t_k)})\bullet \eta&=\\
	\sum_{(t_1,\ldots,t_k)\in F}(\overline{f_1(t_1)}&,\ldots,\overline{f_k(t_k)})\bullet \sum_{i=1}^d(a_i1_0)(1_{A1\alpha_{1i}})\cdots(1_{A_{k\alpha_{ki}}})\\
		=\sum_{(t_1,\ldots,t_k)\in F}&\sum_{i=1}^d(a_i1_0)(\overline{f_1(t_1)}\bullet 1_{A1\alpha_{1i}})\cdots(\overline{f_k(t_k)}\bullet 1_{A_{k\alpha_{ki}}})\\
		=\sum_{i=1}^d(a_i1_0)&(\overline{\sum_{t,\in F_1}f_1(t)}\bullet 1_{A1\alpha_{1i}})\cdots(\overline{\sum_{t\in F_k}f_k(t)})\bullet 1_{A_{k\alpha_{ki}}})
	\end{align*}
For every $F=\times_{i=1}^kF_i\in\times_{i=1}^k P_f(\mathbb{N})$ and for sequences $f_1,\ldots,f_k$ and for every $x\in\V(k,m)$, we have 
\[
T_F^\eta(x)=x+	\sum_{(t_1,\ldots,t_k)\in F}	(\overline{f_1(t_1)},\ldots,\overline{f_k(t_k)})\bullet \eta.
\]

	\begin{theorem} 
 Let  \( p_1, p_2,\ldots, p_m\in \mathbb{Z}[x_1, \ldots, x_k] \) be multivariable polynomials without constant terms. Then for any finite partition $\mathcal{C}$ of \( \mathbb{Z} \) and every $f_1,\ldots,f_k:\mathbb{N}\to \mathbb{Z}$, there exist a cell $C\in\mathcal{C}$, a vector  \(a \) and $F=\times_{i=1}^kF_i\in\times_{i=1}^k P_f(\mathbb{N})$ such that \[
\{a+p_i(\sum_{t\in F_1}f_1(t),\sum_{t\in F_2}f_2(t),\ldots,\sum_{t\in F_k}f_k(t)):i=1,2,\ldots,m\}\subseteq C.
\] 
\end{theorem}
\begin{proof}
	Pick sufficiently number $l\in\mathbb{N}$. So there exist $\eta_1,\ldots,\eta_m\in V(k,l)$ such that $P_{x_1,\ldots,x_k}(\eta_i)=p_i$ for every $i=1,\ldots,m$. Now define 
	\[
	T_F^{\eta_i}(x)=x+	\sum_{(t_1,\ldots,t_k)\in F}	(\overline{f_1(t_1)},\ldots,\overline{f_k(t_k)})\bullet \eta_i
	\] for every $x\in\Ir=\bigcap_{i=1}^m\Ir(\{\eta_i\})$, every $F=\times_{i=1}^kF_i\in\times_{i=1}^k P_f(\mathbb{N})$ and for every $i=1,\ldots,m$. Define 
	\[
	I=\left\{ \left(T_F^{\eta_1}(x), \ldots, T_F^{\eta_m}(x)\right) : x \in \Ir,\,F=\times_{i=1}^kF_i\in\times_{i=1}^k P_f(\mathbb{N}) \right\},
	\]
	and $V=I\cup \Delta_\Ir^k$. Now by Remark \ref{4.77}, Theorem \ref{5.9} and Theorem \ref{van}, our Theorem has been proved.
	\end{proof}

\bibliographystyle{amsplain}

\end{document}